\pdfoutput=1
\documentclass[11pt]{amsart}

\usepackage[margin=1in]{geometry}
\usepackage{amsmath}
\usepackage{amssymb}
\usepackage{mathtools}
\usepackage{mathrsfs}
\usepackage{graphicx}
\usepackage{url}
\usepackage[hidelinks]{hyperref}

\allowdisplaybreaks[2]
\numberwithin{equation}{section}
\providecommand{\tightlist}{%
  \setlength{\itemsep}{0pt}\setlength{\parskip}{0pt}}
\providecommand{\Description}[1]{}

\newtheorem{theorem}{Theorem}[section]
\newtheorem{lemma}[theorem]{Lemma}
\newtheorem{proposition}[theorem]{Proposition}
\newtheorem{corollary}[theorem]{Corollary}
\newtheorem{conjecture}[theorem]{Conjecture}
\theoremstyle{definition}
\newtheorem{definition}[theorem]{Definition}
\newtheorem{example}[theorem]{Example}
\theoremstyle{remark}
\newtheorem{remark}[theorem]{Remark}

\title[Cover-Time Changes Under Edge Addition]{Nonlocality of Cover-Time Changes Under Edge Addition}
\author{Ian Meng Si}
\address{Department of Statistics, University of California, Berkeley, Berkeley, CA 94720, USA}
\email{ianmeng0511@berkeley.edu}
\urladdr{https://orcid.org/0009-0006-9406-9579}
\date{}

\hypersetup{
  pdftitle={Nonlocality of Cover-Time Changes Under Edge Addition},
  pdfauthor={Ian Meng Si}
}

\begin{document}

\begin{abstract}
Let \(G\) be a finite connected simple graph, let \(uv\) be a nonedge, and
let \(s\) be a starting vertex. We study the exact change in the expected
cover time of simple random walk when \(uv\) is inserted. A killed Green
matrix update, combined with target-set inclusion--exclusion, gives an exact
formula using only the original graph. The response can have either sign.

Our main result is a nonlocality theorem. For every radius \(r\ge1\), we
construct two marked configurations whose ambient-degree-labelled
radius-\(r\) neighbourhoods at \(s,u,v\) are isomorphic but whose fixed-start
cover-time responses have opposite signs. The construction also matches the
three marked degrees, the marked distance, and the effective resistance
\(R_{uv}\). The two graphs have different orders; equal-order nonlocality
remains open.

We complement this result with a conductance interpolation theorem and an
occupation interpretation of its high-conductance coefficient. After
contracting \(u\) and \(v\), that coefficient is a positive multiple of the
expected pre-cover occupation of the contracted vertex, and its zero case is
classified exactly. A path with two pendant insertion endpoints is solved for
all starting vertices and shows a sharp change of sign across the start set.
\end{abstract}

\maketitle

\section{Introduction}

The cover time of a finite graph is the first time at which a random
walk has visited every vertex. Its algorithmic importance goes back to
universal traversal sequences \cite{aleliunas1979}; general covering
inequalities were developed by Matthews \cite{matthews1988}, and Feige
obtained the sharp extremal order \cite{feige1995upper,feige1995lower}.
Lov\'asz's survey \cite{lovasz1993} gives classical background.

Electrical-network methods relate hitting and commute times to voltages
and effective resistance \cite{doyle1984,chandra1996,tetali1991}.
Rayleigh monotonicity therefore makes edge addition look benign at the
level of effective resistance. Cover time is different: it is governed
by joint avoidance of many targets, and adding one edge can either
increase or decrease it. Barlow, Ding, Nachmias, and Peres proved that a
single inserted edge can increase the worst-start cover time, though by
at most a universal factor \cite{barlow2011}. Ding, Lee, and Peres later
related cover time, up to universal constants, to the Gaussian free
field and majorizing measures \cite{dingleeperes2012}. Those global
comparison theories do not determine the exact additive response to one
inserted edge.

Exact distributional formulae for cover time on arbitrary graphs were
given by Zlatanov and Kocarev \cite{zlatanov2009}. The present problem
is more specific: we compare the same starting state before and after a
rank-one graph perturbation and ask what information can determine the
sign of that change.

Let \(G=(V,E)\) be a finite connected simple undirected graph, let
\(uv\notin E\), and write \(G^+=G+uv\). For a start \(s\), set

\begin{equation}
t_{\rm cov}(G,s)=\mathbb E_s\tau_{\rm cov},
\qquad
\delta_s(G;uv)=t_{\rm cov}(G^+,s)-t_{\rm cov}(G,s).
\label{eq:1.1}
\end{equation}

The sign of \(\delta_s\) is the fixed-start response. The first
contribution of this paper is an old-graph-only formula for it. The
formula uses target-set inclusion--exclusion and a Sherman--Morrison
update of grounded Laplacians \cite{sherman1950}; its exponential size
is an exact representation, not a complexity claim.

The principal result shows that no bounded marked neighbourhood decides
the sign. We use the following precise convention. The radius-\(r\)
local datum of \((G;s,u,v)\) is the isomorphism class of the induced
subgraph on vertices at distance at most \(r\) from \(\{s,u,v\}\), with
\(s,u,v\) individually marked and every visible vertex labelled by its
degree in the ambient graph. A rule receiving this datum is not also
given the graph order or other global scalars. The theorem additionally
matches several such scalars, as stated below.

\begin{theorem}[arbitrary-radius nonlocality]\label{thm:1.1} For every integer
\(r\ge1\), there are finite connected simple rooted insertion
configurations \((G_r^-;s,u,v)\) and \((G_r^+;s,u,v)\) such that \(uv\)
is absent in both, their ambient-degree-labelled marked radius-\(r\)
data are isomorphic, and

\begin{equation}
\delta_s(G_r^+;uv)>0>\delta_s(G_r^-;uv).
\label{eq:1.2}
\end{equation}

The construction also matches

\begin{equation}
d_s=d_u=d_v=1,\qquad
\operatorname{dist}(s,\{u,v\})=2,\qquad R_{uv}=2.
\label{eq:1.3}
\end{equation}
\end{theorem}

The two graphs have different orders. Thus Theorem~\ref{thm:1.1} excludes rules
based on the stated local datum, even if augmented by the matched
scalars in \eqref{eq:1.3}, but it does not exclude a rule using local data
together with graph order. Equal-order nonlocality remains open.

A second structural result concerns conductance interpolation. Give the
new edge conductance \(\lambda\ge0\) and write
\(f_s(\lambda)=t_{\rm cov}(G+\lambda uv,s)\). Each target contribution
is quadratic-over-linear, and the global response divided by \(\lambda\)
is a finite sum of simple poles on the negative real axis. Its
coefficient at infinity has a probabilistic meaning: after contracting
\(u,v\) to \(w\), it is \(2/(d_u+d_v)\) times the expected number of
visits to \(w\) before cover. This coefficient is nonnegative, and we
classify its zero case exactly.

Finally, we solve a path with two pendant insertion endpoints for every
start. That family gives both signs within one graph and proves
unit-conductance strictness there. Worst-start relocation and the
unresolved global equality questions are kept as consequences and open
problems rather than used as the main narrative.

The component tools have classical origins: target-set
inclusion--exclusion is classical in spirit, and Sherman--Morrison is a
standard rank-one inverse identity. The contribution here is their
combination into an exact old-graph-only perturbation formula and the
nonlocality and contraction--occupation consequences derived from it. We
do not claim either ingredient alone as a principal novelty.

The paper is organised as follows. Section 2 derives the exact
perturbation formula. Section 3 proves Theorem~\ref{thm:1.1}. Section 4 develops
conductance interpolation and the contraction--occupation theorem.
Section 5 treats the solved pendant-path family. Section 6 records the
short worst-start consequence and the delimited equality evidence.
Section 7 concludes with open problems.

\section{Exact fixed-start perturbation}

\subsection{Target-set inclusion--exclusion}

For each nonempty \(S\subseteq V\setminus\{s\}\), define

\[
T_S=\inf\{t\ge0:X_t\in S\},
\qquad
H_s(S)=\mathbb E_sT_S.
\]

Finite inclusion-exclusion applied at each time, followed by the
convergent tail sum, gives

\begin{equation}
\boxed{
t_{\mathrm{cov}}(G,s)
=
\sum_{\varnothing\neq S\subseteq V\setminus\{s\}}
(-1)^{|S|+1}H_s(S).
}
\label{eq:2.1}
\end{equation}

\begin{lemma}[target-set inclusion-exclusion]\label{lem:2.1} Formula \eqref{eq:2.1}
holds.

\end{lemma}

\begin{proof} For \(x\ne s\), write \(A_x(t)=\{T_x>t\}\). The walk has
not covered the graph by time \(t\) precisely on
\(\bigcup_{x\ne s}A_x(t)\). Finite inclusion-exclusion gives

\[
\mathbb P_s(\tau_{\mathrm{cov}}>t)
=\sum_{\varnothing\ne S\subseteq V\setminus\{s\}}
(-1)^{|S|+1}\mathbb P_s(T_S>t).
\]

Summing over \(t\ge0\) and using
\(\mathbb E T=\sum_{t\ge0}\mathbb P(T>t)\) proves \eqref{eq:2.1}. The exchange is
legitimate because the subset sum is finite and all hitting times have a
geometric tail: in every block of \(2|E|\) steps there is a
graph-dependent positive probability of following a fixed spanning-tree
traversal. 
\end{proof}

\subsection{Killed Green matrices}

Fix nonempty \(S\subseteq V\setminus\{s\}\) and put \(B=V\setminus S\).
Then

\[
L_B:=L[B,B]\succ0.
\]

Indeed, extending \(x\in\mathbb R^B\) by zero on \(S\) gives

\[
x^{\mathsf T}L_Bx
=
\sum_{\{p,q\}\in E}(\widetilde x_p-\widetilde x_q)^2,
\]

which vanishes only for \(x=0\). Hence

\[
K_S:=L_B^{-1}
\]

exists. Moreover,

\[
K_S
=(I-P_{BB})^{-1}D_B^{-1}
=
\sum_{t\ge0}P_{BB}^{t}D_B^{-1}\ge0
\]

entrywise. Its probabilistic interpretation is

\begin{equation}
(K_S)_{xy}
=
\frac1{d_y}
\mathbb E_x\!\left[
\sum_{t=0}^{T_S-1}\mathbf1_{\{X_t=y\}}
\right].
\label{eq:2.2}
\end{equation}

The set-hitting vector is

\begin{equation}
h_B^S=K_Sd_B.
\label{eq:2.3}
\end{equation}

After the insertion,

\[
L_B'=L_B+b_Bb_B^{\mathsf T},
\qquad
d_B'=d_B+a_B,
\]

and therefore

\begin{equation}
(K_S)'
=
K_S-
\frac{K_Sb_Bb_B^{\mathsf T}K_S}
{1+b_B^{\mathsf T}K_Sb_B}.
\label{eq:2.4}
\end{equation}

It follows that

\begin{equation}
\boxed{
F_s(S):=H_s'(S)-H_s(S)
=
e_{s,B}^{\mathsf T}K_Sa_B
-
\frac{
(e_{s,B}^{\mathsf T}K_Sb_B)
\bigl(b_B^{\mathsf T}K_S(d_B+a_B)\bigr)
}{1+b_B^{\mathsf T}K_Sb_B}.
}
\label{eq:2.5}
\end{equation}

\begin{lemma}[killed Green update]\label{lem:2.2} For nonempty
\(S\subseteq V\setminus\{s\}\), the matrix \(L_B\) is positive definite,
\(h_B^S=K_Sd_B\), and \eqref{eq:2.4}--\eqref{eq:2.5} hold.

\end{lemma}

\begin{proof} Extending \(x\in\mathbb R^B\) by zero on \(S\) gives the
displayed Dirichlet-energy identity, so connectedness implies
\(L_B\succ0\). Since \(L_B=D_B(I-P_{BB})\),

\[
K_S=(I-P_{BB})^{-1}D_B^{-1}
=\sum_{t\ge0}P_{BB}^tD_B^{-1}.
\]

The first-step equations for \(h_x^S=\mathbb E_xT_S\) are
\(h_x^S=1+\sum_{y\in B}P_{xy}h_y^S\), equivalently \(L_Bh_B^S=d_B\);
hence \(h_B^S=K_Sd_B\). The inserted edge changes the killed Laplacian
and degree vector to \(L_B+b_Bb_B^{\mathsf T}\) and \(d_B+a_B\).
Sherman-Morrison gives \eqref{eq:2.4}. Multiplying \eqref{eq:2.4} by \(d_B+a_B\),
subtracting \(K_Sd_B\), and taking coordinate \(s\) gives \eqref{eq:2.5}.

\end{proof}

Every quantity on the right belongs to \(G\).

\subsection{The perturbation theorem}

\begin{theorem}[exact fixed-start perturbation]\label{thm:2.3} Combining \eqref{eq:2.1}
and \eqref{eq:2.5},

\begin{equation}
\boxed{
\delta_s
:=
t_{\mathrm{cov}}(G^+,s)-t_{\mathrm{cov}}(G,s)
=
\sum_{\varnothing\neq S\subseteq V\setminus\{s\}}
(-1)^{|S|+1}F_s(S).
}
\label{eq:2.6}
\end{equation}

This representation contains \(2^{n-1}-1\) subset terms. That count
describes this inclusion-exclusion representation only; no minimality
claim is made.

\end{theorem}

\begin{proof} Apply Lemma~\ref{lem:2.1} to \(G^+\) and \(G\), subtract the two
finite sums, and use Lemma~\ref{lem:2.2} for each target set \(S\). Every matrix
in the result is a killed Green matrix of the old graph. 
\end{proof}

\subsection{Endpoint location and the alternating-sign obstruction}

If \(u,v\in S\), then

\begin{equation}
F_s(S)=0.
\label{eq:2.7}
\end{equation}

If exactly one endpoint is outside \(S\), call it \(w\). Then

\begin{equation}
\boxed{
F_s(S)
=
\frac{K_{sw}(1-h_w^S)}{1+K_{ww}}
\le0.
}
\label{eq:2.8}
\end{equation}

If \(u,v\notin S\), then

\begin{equation}
\boxed{
\begin{aligned}
F_s(S)
={}&K_{su}+K_{sv}\\
&-
\frac{(K_{su}-K_{sv})
\bigl(h_u^S-h_v^S+K_{uu}-K_{vv}\bigr)}
{1+K_{uu}+K_{vv}-2K_{uv}}.
\end{aligned}
}
\label{eq:2.9}
\end{equation}

The latter has no universal sign.

\begin{proposition}[endpoint-location formulae]\label{prop:2.4} Equations
\eqref{eq:2.7}--\eqref{eq:2.9} hold.

\end{proposition}

\begin{proof} If \(u,v\in S\), then \(a_B=b_B=0\), so \eqref{eq:2.5} vanishes. If
exactly one endpoint, say \(w\), lies in \(B\), then \(a_B=e_w\) and
\(b_B=\pm e_w\). Substitution into \eqref{eq:2.5}, together with
\(h_w^S=e_w^{\mathsf T}K_Sd_B\), gives

\[
F_s(S)=K_{sw}-
\frac{K_{sw}(h_w^S+K_{ww})}{1+K_{ww}}
=\frac{K_{sw}(1-h_w^S)}{1+K_{ww}}.
\]

Because \(w\notin S\), at least one step is required to hit \(S\), so
\(h_w^S\ge1\); also \(K_{sw}\ge0\). This proves \eqref{eq:2.8}. If both endpoints
lie in \(B\), then \(a_B=e_u+e_v\), \(b_B=e_u-e_v\), and direct
expansion of \eqref{eq:2.5} gives \eqref{eq:2.9}. 
\end{proof}

The sign in \eqref{eq:2.8} does not imply \(\delta_s\le0\): even-cardinality
target sets enter \eqref{eq:2.6} with a minus sign and therefore reverse the
contribution.

Let \(\mathcal O_k\) and \(\mathcal E_k\) be the odd- and
even-cardinality target sets satisfying \(|S\cap\{u,v\}|=k\). Put

\[
q_s(S)
=
\frac{K_{sw}(h_w^S-1)}{1+K_{ww}}
\ge0
\qquad(S\in\mathcal O_1\cup\mathcal E_1),
\]

and let \(g_s^{(0)}(S)\) denote the right side of \eqref{eq:2.9}. Then

\begin{equation}
\boxed{
\delta_s=G_{0,s}-Q_{1,s},
}
\label{eq:2.10}
\end{equation}

where

\[
G_{0,s}
=
\sum_{S\in\mathcal O_0}g_s^{(0)}(S)
-
\sum_{S\in\mathcal E_0}g_s^{(0)}(S),
\]

\[
Q_{1,s}
=
\sum_{S\in\mathcal O_1}q_s(S)
-
\sum_{S\in\mathcal E_1}q_s(S).
\]

Therefore

\begin{equation}
\boxed{
\delta_s\lessgtr0
\iff
G_{0,s}\lessgtr Q_{1,s}.
}
\label{eq:2.11}
\end{equation}

This is the exact arbitrary-graph sign criterion. It is algebraic and
generally exponential-size, not a broadly local topological
characterisation.

\section{Arbitrary-radius nonlocality}

\subsection{Exact construction}

Fix an integer \(r\ge1\), and define

\begin{equation}
M=36(r+1)^2,
\qquad
R=M(M+1)+5.
\label{eq:3.1.1}
\end{equation}

The constants are chosen for transparent uniform inequalities and are
not optimised.

For integers \(M,L\ge1\), let

\[
Q=\{q_1,\ldots,q_M\},
\qquad
P_L=\{p_1,\ldots,p_L\},
\]

all disjoint from the four named vertices \(c,s,u,v\). Define
\(X_{M,L}\) by

\begin{equation}
V(X_{M,L})
=\{c,s,u,v\}\sqcup Q\sqcup P_L
\label{eq:3.1.2}
\end{equation}

and

\begin{equation}
\boxed{
\begin{aligned}
E(X_{M,L})={}&
\binom{Q\cup\{c\}}2
\cup\{cs,cu,cv\}\\
&\cup\{cp_1\}
\cup\{p_jp_{j+1}:1\le j<L\}.
\end{aligned}}
\label{eq:3.1.3}
\end{equation}

Thus \(c\) \textbf{is one of the vertices of the \(K_{M+1}\) core}: the
induced graph on \(Q\cup\{c\}\) is complete. The vertices \(s,u,v\) are
three pendant leaves, each attached directly to \(c\). The additional
path is pendant at \(c\): it is

\[
c-p_1-p_2-\cdots-p_L.
\]

Throughout this paper, ``path length \(L\)'' means exactly \(L\) edges;
the path contributes the \(L\) new vertices \(p_1,\ldots,p_L\).

The two rooted insertion configurations are

\begin{equation}
(G_r^-;s,u,v)=(X_{M,r};s,u,v),
\qquad
(G_r^+;s,u,v)=(X_{M,R};s,u,v).
\label{eq:3.1.4}
\end{equation}

The starting vertex is \(s\), and the inserted edge is the genuine
nonedge \(uv\). Their exact orders are

\begin{equation}
\boxed{
|V(G_r^-)|=M+r+4,
\qquad
|V(G_r^+)|=M+R+4.
}
\label{eq:3.1.5}
\end{equation}

Since \(R>r\), the orders are not equal.

The old edge count of \(X_{M,L}\) is

\begin{equation}
m(M,L)=\binom{M+1}{2}+L+3
=\frac{M(M+1)}2+L+3.
\label{eq:3.1.6}
\end{equation}

\subsection{\texorpdfstring{Formal radius-\(r\) local
equivalence}{Formal radius-r local equivalence}}

\begin{definition}[local information]\label{def:3.1} For a marked configuration
\((G;s,u,v)\), its radius-\(r\) local datum is the isomorphism class of
the induced subgraph on

\[
B_r(G;\{s,u,v\})
=\{x:\operatorname{dist}_G(x,\{s,u,v\})\le r\},
\]

with \(s,u,v\) preserved as three individual marks and every visible
vertex labelled by its degree in the full ambient graph. Unless
explicitly stated, this datum contains neither the graph order nor any
other global scalar.
\end{definition}

Let \(\mathcal M=\{s,u,v\}\). In either graph, the exact distances from
the marked set are

\begin{equation}
\operatorname{dist}(x,\mathcal M)=
\begin{cases}
0,&x\in\{s,u,v\},\\
1,&x=c,\\
2,&x\in Q,\\
j+1,&x=p_j.
\end{cases}
\label{eq:3.2.1}
\end{equation}

These distances follow directly from \eqref{eq:3.1.3}: every marked vertex is
pendant at \(c\), every \(q_i\) is adjacent to \(c\), and the only route
from \(p_j\) to the marked set passes through the \(j\)-edge path to
\(c\) and then one marked edge.

\subsubsection{The boundary radius one}

Equation \eqref{eq:3.2.1} gives

\[
V(B_1)=\{s,u,v,c\}
\]

in both configurations. The induced graph is the three-leaf claw with
centre \(c\). The ambient degrees in the full graph are

\begin{equation}
d_s=d_u=d_v=1,
\qquad
d_c=M+4,
\label{eq:3.2.2}
\end{equation}

because \(c\) is adjacent to the \(M\) clique vertices, the three marked
leaves, and \(p_1\). Hence the identity on \(\{s,u,v,c\}\) is an
isomorphism of marked, ambient-degree-labelled radius-one balls. Neither
\(Q\) nor \(p_1\) is visible when \(r=1\).

\subsubsection{Radii at least two}

Both path lengths in \eqref{eq:3.1.4} are at least \(r\). Hence \eqref{eq:3.2.1} gives, in
both graphs,

\begin{equation}
V(B_r)
=\{s,u,v,c\}
\cup Q
\cup\{p_1,\ldots,p_{r-1}\}.
\label{eq:3.2.3}
\end{equation}

The identity on the named vertices preserves the induced edges: it
identifies the complete core \(Q\cup\{c\}\), the three marked pendant
edges, and the visible initial path segment

\[
c-p_1-\cdots-p_{r-1}.
\]

The ambient degrees are exactly

\begin{equation}
\boxed{
\begin{aligned}
d_s=d_u=d_v&=1,\\
d_c&=M+4,\\
d_{q_i}&=M\quad(1\le i\le M),\\
d_{p_j}&=2\quad(1\le j\le r-1).
\end{aligned}}
\label{eq:3.2.4}
\end{equation}

The boundary label requires care. The last visible path vertex is
\(p_{r-1}\). It has an edge to the invisible vertex \(p_r\) in both
configurations: in \(G_r^-\), \(p_r\) is the terminal path vertex, while
in \(G_r^+\) it is an internal path vertex. Thus \(d_{p_{r-1}}=2\) in
both full graphs. The label is not mistakenly computed inside the
induced ball.

Equations \eqref{eq:3.2.2}--\eqref{eq:3.2.4} prove the required marked,
ambient-degree-labelled isomorphism

\begin{equation}
\boxed{
B_r(G_r^-;\{s,u,v\})
\cong
B_r(G_r^+;\{s,u,v\}).
}
\label{eq:3.2.5}
\end{equation}

The prospective nonedge \(uv\) and the individual labels \(s,u,v\) are
preserved.

\begin{figure}
\centering
\includegraphics[width=0.94\textwidth,keepaspectratio]{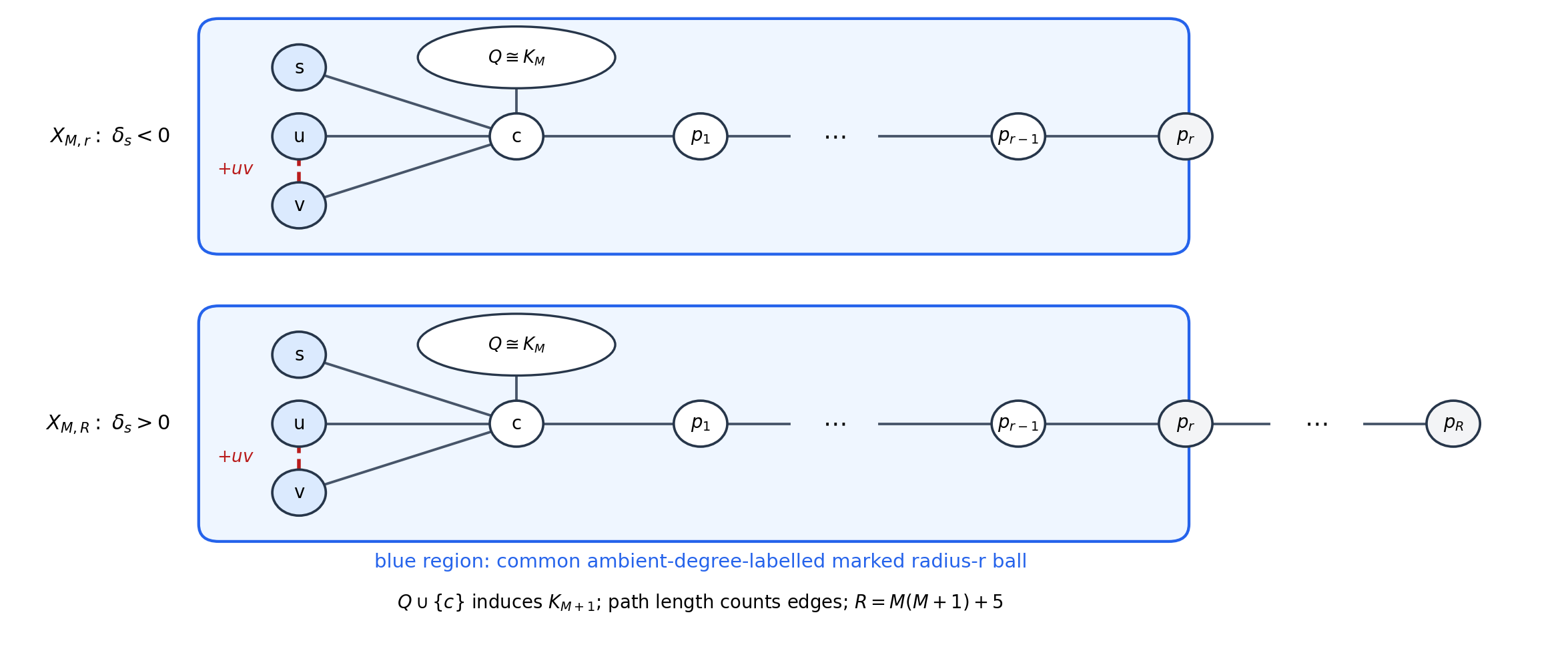}
\caption{Arbitrary-radius construction. The blue region is the common ambient-degree-labelled marked radius-$r$ ball; the terminal distinguishing structure lies outside it.}
\label{fig:arbitrary-radius-construction}
\Description{Two horizontal graph constructions share the same blue radius-$r$ neighbourhood around the marked leaves $s$, $u$, and $v$. The upper graph ends at $p_r$ and has negative cover-time response; the lower graph continues to $p_R$ beyond the blue region and has positive response.}
\end{figure}

\subsection{Matched scalar invariants}

The marked degrees have already been computed:

\[
d_s=d_u=d_v=1
\]

in both graphs. Since each pair of distinct marked leaves is joined
through \(c\) by a two-edge path,

\begin{equation}
\operatorname{dist}(s,\{u,v\})=2.
\label{eq:3.3.1}
\end{equation}

For effective resistance, the edges \(uc\) and \(cv\) are
unit-resistance bridges. Every other part of the graph is attached to
the \(u\)-\(v\) network only at the single terminal \(c\). In the unit
\(u\)-to-\(v\) Dirichlet problem, the harmonic extension on every such
one-terminal attached network is constant at the potential of \(c\), so
it carries no current. Equivalently, Thomson's principle deletes these
zero-current branches. Therefore

\begin{equation}
\boxed{R_{uv}=R_{uc}+R_{cv}=1+1=2}
\label{eq:3.3.2}
\end{equation}

in both configurations.

Thus all the requested extra invariants except order agree. The unequal
orders \eqref{eq:3.1.5} are an explicit limitation of the theorem.

\subsection{Exact cut-vertex transfer theorem}

The response separates into two physical effects. First, the old and new
walks accumulate full-excursion reward until both the local insertion
gadget and the remote graph have been covered. Second, full-excursion
accounting overshoots the actual cover time and must subtract the unused
suffix of the excursion in which coverage is completed.

Artificial Poisson time records only the order in which excursions leave
\(c\). It makes the local and remote completion processes independent
and allows the accumulated full-excursion reward and the terminal unused
suffix to be handled separately. Physical random-walk time always
remains the sum of discrete excursion lengths.

\begin{theorem}[local-pair insertion against an arbitrary remote graph]\label{thm:3.2}

Let \(H\) be a finite connected graph with distinguished vertex \(c\),
and assume \(W=V(H)\setminus\{c\}\ne\varnothing\). Form \(G\) by
adjoining three new vertices \(s,u,v\) and the three edges \(cs,cu,cv\),
with no other edges incident to \(s,u,v\). Let

\[
G'=G+uv,
\qquad
m=|E(G)|,
\]

and start at \(s\).

Work on a product probability space. At \(c\), give every directed
departure edge \((c,z)\) an independent rate-one Poisson process. At
each ring, attach an independent mark having the law of a complete
discrete simple-random-walk excursion that starts with \(c\to z\) and
ends on its first return to \(c\). Marks belonging to distinct rings are
independent. The superposed ring order is therefore an i.i.d. uniform
sequence of departure edges, exactly matching successive departures from
\(c\) in the discrete walk.

Let:

\begin{itemize}
\tightlist
\item
  \(T\) be the artificial time at which the remote vertices \(W\) have
  all been visited by remote excursion marks;
\item
  \(A\) be the artificial completion time of the old local pair
  \(\{u,v\}\);
\item
  \(B\) be the artificial completion time of the new local triangle
  \(cuv\);
\item
  \(R_T^{\mathrm{rem}}\) be the physical length of the unused suffix of
  the terminal remote excursion, from the instant at which the last
  unvisited vertex of \(W\) is first visited until return to \(c\);
\item
  \(\mathscr G_T\) be the sigma-field generated by the remote artificial
  ring times, departure labels, and excursion paths only through that
  first completion instant; and
\item
  \(J=\mathbb E[R_T^{\mathrm{rem}}\mid\mathscr G_T]\), the conditional
  expected unused remote suffix.
\end{itemize}

Use independent local marked processes to define \(A\) and \(B\), each
independent of the common remote process defining \(T,J\). Put

\begin{equation}
a=\mathbb E e^{-T},
\qquad
b=\mathbb E e^{-3T/2}.
\label{eq:3.4.1}
\end{equation}

Then \(T,A,B,J\) are integrable, \(a,b\) are finite, and

\begin{equation}
\boxed{
\delta_s(G;uv)
=2\mathbb ET
-2(2m-1)\left(a-\frac23b\right)
-2\mathbb E\!\left[J\left(e^{-T}-e^{-3T/2}\right)\right].
}
\label{eq:3.4.2}
\end{equation}

\end{theorem}

\begin{proof}

\textbf{Step 1: filtration and stopping times.}

Let \(\mathscr F_t\) be generated by all Poisson ring times up to
artificial time \(t\), their departure-edge labels, and the
discovered-set/trajectory marks of the corresponding excursions. No
future mark or future excursion-length reward is exposed. The events
\(\{T\le t\}\), \(\{A\le t\}\), and \(\{B\le t\}\) are determined by the
union of the revealed discovered sets, so \(T,A,B\) are stopping times.
Their maxima

\[
Z_0=\max(T,A),
\qquad
Z_1=\max(T,B)
\]

are stopping times as well.

More explicitly, the construction is made on the product space of the
remote process and two independent local processes. Marks are revealed
chronologically. The coverage rules use an excursion path only to
determine which vertices have appeared; they never use an unrevealed
suffix or the numerical value of a length reward to decide whether to
stop. Thus \(T,A,B,Z_0,Z_1\) are stopping times for the marked-process
filtration, not anticipative functionals of future rewards.

All are integrable. To see this without an asymptotic argument, choose
finitely many particular \(c\)-to-\(c\) excursion trajectories whose
combined remote visited sets cover \(W\). A prescribed finite word of
departure labels and these finite trajectories has some probability
\(p>0\) of occurring in any block of sufficiently many successive
departure events. Hence the number of blocks needed to cover \(W\) is
stochastically dominated by a geometric random variable. The artificial
interarrival times have finite means. Thus \(\mathbb E T<\infty\). The
local completion times have the explicit exponential tails below, and
therefore \(Z_0,Z_1\) are integrable. Every finite-state return
excursion also has finite mean physical length. Moreover, the unused
terminal suffix is bounded by the sum of the full excursion lengths
through \(Z_0\) or \(Z_1\); Step 2 proves that this sum has finite
expectation. Thus \(R_T^{\mathrm{rem}}\), and hence \(J\), is
integrable.

The stopping rules use only revealed event times and coverage marks.
They do not inspect any unrevealed future excursion reward.

\textbf{Step 2: stopped full-excursion reward.}

Let \(d_c\) be the degree of \(c\). In \(G\), the superposed departure
process has rate \(d_c\). Write the ordered excursion marks as

\[
\Xi_i=(C_i,L_i),
\]

where \(C_i\) records the departure label and all coverage information
in the \(i\)-th excursion, and \(L_i\ge0\) is its complete physical
length. The marks \((\Xi_i)_{i\ge1}\) are i.i.d. Let
\(\mathcal F_i=\sigma(\Xi_1,\ldots,\Xi_i)\).

\begin{lemma}[stopped marked-excursion reward]\label{lem:3.3} Let
\(\Gamma_n\in\sigma(C_1,\ldots,C_n)\) be an increasing sequence of
completion events, with \(\Gamma_0=\varnothing\), and let
\(N=\inf\{n\ge1:\Gamma_n\text{ occurs}\}\). If \(L_i\ge0\) almost
surely, \(\mathbb EN<\infty\), and \(\mathbb EL_1<\infty\), then

\begin{equation}
\mathbb E\sum_{i=1}^{N}L_i=\mathbb EL_1\,\mathbb EN.
\label{eq:3.4.3}
\end{equation}

\end{lemma}

\begin{proof} Since \(\Gamma_n\in\mathcal F_n\), the index \(N\) is an
\((\mathcal F_i)\)-stopping time. Moreover, for every \(i\ge1\),

\[
\{N\ge i\}=\Gamma_{i-1}^{\mathsf c}\in\mathcal F_{i-1}.
\]

Because the new mark \(\Xi_i\), and hence \(L_i\), is independent of
\(\mathcal F_{i-1}\),

\[
\mathbb E[L_i\mathbf1_{\{N\ge i\}}]
=\mathbb EL_1\,\mathbb P(N\ge i).
\]

Tonelli's theorem now gives

\[
\mathbb E\sum_{i=1}^{N}L_i
=\sum_{i\ge1}\mathbb E[L_i\mathbf1_{\{N\ge i\}}]
=\mathbb EL_1\sum_{i\ge1}\mathbb P(N\ge i)
=\mathbb EL_1\,\mathbb EN.
\]

The terminal excursion length \(L_N\) may be strongly correlated with
the event that coverage is completed at excursion \(N\). Wald's identity
does not require their independence; it uses the independence of the new
mark \(L_i\) from the pre-\(i\) survival event \(\{N\ge i\}\). This is
the elementary stopped-sum argument behind Wald's identity
\cite{gut2009}. 
\end{proof}

\begin{example}[within-mark correlation]\label{ex:3.4} Suppose each mark
succeeds with probability \(p\), a successful mark has length \(100\), a
failed mark has length \(1\), and \(N\) is the first successful mark.
Then \(L_N=100\) identically, so terminal length and completion are
maximally correlated. Nevertheless,

\[
\mathbb E\sum_{i=1}^{N}L_i
=\frac{1-p}{p}+100
=\frac1p+99
=(1+99p)\frac1p
=\mathbb EL_1\,\mathbb EN.
\]
\end{example}

Return-time Kac's formula \cite{norris1997} gives

\begin{equation}
\mathbb EL_1=\frac1{\pi(c)}=\frac{2m}{d_c}.
\label{eq:3.4.4}
\end{equation}

Apply Lemma~\ref{lem:3.3} with \(N=N_Z\), the number of ordered departures
required to reach an integrable coverage stopping time \(Z\) of Step 1.
The geometric-block argument gives \(\mathbb EN_Z<\infty\). Separately,
let \(E_1,E_2,\ldots\) be the interarrival spacings of the superposed
Poisson process. They are i.i.d. \(\operatorname{Exp}(d_c)\) and are
independent of the entire ordered mark sequence. The index \(N_Z\)
depends only on the ordered departure labels and excursion coverage
marks, not on these spacings. Consequently,

\begin{equation}
\mathbb EZ
=\mathbb E\sum_{i=1}^{N_Z}E_i
=\frac{\mathbb EN_Z}{d_c},
\qquad
\mathbb EN_Z=d_c\mathbb EZ.
\label{eq:3.4.5}
\end{equation}

This Poisson-spacing calculation is independent of Lemma~\ref{lem:3.3} and does
not inspect any excursion length. Combining \eqref{eq:3.4.3}--\eqref{eq:3.4.5},

\begin{equation}
\mathbb E\sum_{i=1}^{N_Z}L_i=2m\mathbb EZ.
\label{eq:3.4.5a}
\end{equation}

In \(G'\), the degree of \(c\) is unchanged and the edge count is
\(m+1\), so the corresponding stopped reward is

\begin{equation}
2(m+1)\mathbb EZ.
\label{eq:3.4.5b}
\end{equation}

This is the precise compensation statement needed in the proof.

\textbf{Step 3: local completion laws and terminal corrections.}

Before insertion, the first local departures along \(cu\) and \(cv\)
occur at independent \(\operatorname{Exp}(1)\) times. Therefore

\begin{equation}
\mathbb P(A>t)=2e^{-t}-e^{-2t}.
\label{eq:3.4.6}
\end{equation}

After insertion, while no endpoint has been seen, departures along
\(cu\) or \(cv\) occur at total rate \(2\). From the first endpoint, the
next triangle step goes either to \(c\) or to the other endpoint with
equal probabilities. Hence the continuous-time phase chain jumps from
the no-endpoint state to the one-endpoint state at rate \(1\), and
directly to completion at rate \(1\). From the one-endpoint state, a
departure directly to the unseen endpoint completes at rate \(1\), while
a departure to the seen endpoint crosses to the unseen endpoint before
returning to \(c\) with probability \(1/2\); the total completion hazard
is therefore \(3/2\). If \(p_0(t),p_1(t)\) are the probabilities of the
two transient phases, then

\[
p_0'(t)=-2p_0(t),
\qquad
p_1'(t)=p_0(t)-\frac32p_1(t),
\qquad
(p_0(0),p_1(0))=(1,0).
\]

Thus \(p_0(t)=e^{-2t}\), \(p_1(t)=2(e^{-3t/2}-e^{-2t})\), and

\begin{equation}
\mathbb P(B>t)=2e^{-3t/2}-e^{-2t}.
\label{eq:3.4.7}
\end{equation}

Full-excursion accounting must subtract the unused suffix of the
terminal excursion. If the old local pair finishes last, the last leaf
is first visited after the outward step and the unused return suffix has
length \(1\). If the new local pair finishes last, completion occurs at
\(u\) or \(v\); the expected hitting time of \(c\) from either endpoint
in the triangle \(cuv\) is the solution of \(h=1+h/2\), namely \(2\). If
the remote graph finishes last, conditional expectation with respect to
\(\mathscr G_T\), followed by the tower property, replaces the actual
unused suffix by \(J\). The strong Markov property at the physical
first-completion instant justifies this terminal-return expectation.

The remote and local artificial processes are independent and have
continuous event-time laws, so ties have probability zero. Equations
\eqref{eq:3.4.5a}--\eqref{eq:3.4.7} therefore give the exact old and new cover times

\begin{equation}
\begin{aligned}
C_0={}&1+2m\,\mathbb E\max(T,A)
-\mathbb P(A>T)
-\mathbb E[J\mathbf1_{\{T>A\}}],\\
C_1={}&1+2(m+1)\,\mathbb E\max(T,B)
-2\mathbb P(B>T)
-\mathbb E[J\mathbf1_{\{T>B\}}].
\end{aligned}
\label{eq:3.4.8}
\end{equation}

The leading (1) is the forced initial step \(s\to c\).

\textbf{Step 4: exact algebra.}

For any nonnegative \(X\perp T\), Tonelli's theorem gives

\begin{equation}
\mathbb E\max(T,X)
=\mathbb ET+
\int_0^\infty\mathbb P(X>t)\mathbb P(T\le t)\,dt,
\label{eq:3.4.9}
\end{equation}

and, for \(\lambda>0\),

\begin{equation}
\int_0^\infty e^{-\lambda t}\mathbb P(T\le t)\,dt
=\frac1\lambda\mathbb E e^{-\lambda T}.
\label{eq:3.4.10}
\end{equation}

Write \(c_2=\mathbb E e^{-2T}\). Substitution of \eqref{eq:3.4.6}--\eqref{eq:3.4.7} into
\eqref{eq:3.4.9}--\eqref{eq:3.4.10} yields

\begin{equation}
\begin{aligned}
\mathbb E\max(T,A)&=\mathbb ET+2a-\frac12c_2,\\
\mathbb E\max(T,B)&=\mathbb ET+\frac43b-\frac12c_2,\\
\mathbb P(A>T)&=2a-c_2,\\
\mathbb P(B>T)&=2b-c_2.
\end{aligned}
\label{eq:3.4.11}
\end{equation}

Conditioning on the remote process through \(T\), and using
\eqref{eq:3.4.6}--\eqref{eq:3.4.7},

\begin{equation}
\begin{aligned}
&\mathbb E\!\left[J
(\mathbf1_{\{T>B\}}-\mathbf1_{\{T>A\}})
\right]\\
&\qquad=
\mathbb E\!\left[J
(\mathbb P(B<T\mid T)-\mathbb P(A<T\mid T))
\right]\\
&\qquad=2\mathbb E\!\left[J(e^{-T}-e^{-3T/2})\right].
\end{aligned}
\label{eq:3.4.12}
\end{equation}

Subtract \(C_0\) from \(C_1\), insert \eqref{eq:3.4.11}--\eqref{eq:3.4.12}, and collect
coefficients. The \(c_2\)-terms cancel, the coefficient of \(a\) is
\(-2(2m-1)\), and the coefficient of \(b\) is \(4(2m-1)/3\). This is
exactly \eqref{eq:3.4.2}. 
\end{proof}

\subsection{Exact specialization to the clique-path graphs}

Fix \(M,L\ge1\), and apply Theorem~\ref{thm:3.2} to \(X_{M,L}\). Let:

\begin{itemize}
\tightlist
\item
  \(\Theta_L\) be the artificial time at which an excursion down the
  pendant path first reaches \(p_L\);
\item
  \(Q_M\) be the artificial time at which all clique vertices
  \(q_1,\ldots,q_M\) have been visited.
\end{itemize}

A departure along \(cp_1\) reaches \(p_L\) before returning to \(c\)
with gambler's-ruin probability \(1/L\). Poisson thinning gives

\begin{equation}
\Theta_L\sim\operatorname{Exp}(1/L).
\label{eq:3.5.1}
\end{equation}

The path and clique clocks are disjoint, so \(\Theta_L\perp Q_M\), and
the remote completion time is exactly

\begin{equation}
T_{M,L}=\max(\Theta_L,Q_M).
\label{eq:3.5.2}
\end{equation}

If the path completes last, completion occurs at \(p_L\), whose expected
hitting time to \(c\) is \(L^2\). If the clique completes last,
completion occurs at some \(q_i\), whose expected hitting time to \(c\)
in \(K_{M+1}\) is \(M\). Hence the exact terminal correction is

\begin{equation}
J_{M,L}
=L^2\mathbf1_{\{\Theta_L>Q_M\}}
+M\mathbf1_{\{Q_M>\Theta_L\}}.
\label{eq:3.5.3}
\end{equation}

Define

\begin{equation}
a_{M,L}=\mathbb E e^{-T_{M,L}},
\qquad
b_{M,L}=\mathbb E e^{-3T_{M,L}/2}.
\label{eq:3.5.4}
\end{equation}

Equations \eqref{eq:3.1.6} and \eqref{eq:3.4.2} give the exact finite identity

\begin{equation}
\boxed{
\begin{aligned}
\delta(M,L)={}&2\mathbb ET_{M,L}
-2\bigl(M(M+1)+2L+5\bigr)
\left(a_{M,L}-\frac23b_{M,L}\right)\\
&-2\mathbb E\!\left[
J_{M,L}
\left(e^{-T_{M,L}}-e^{-3T_{M,L}/2}\right)
\right].
\end{aligned}}
\label{eq:3.5.5}
\end{equation}

There is no asymptotic replacement in \eqref{eq:3.5.5}. The two perturbations
used in the theorem are exactly

\begin{equation}
\delta_s(G_r^-;uv)=\delta(M,r),
\qquad
\delta_s(G_r^+;uv)=\delta(M,R),
\label{eq:3.5.6}
\end{equation}

with \(M,R\) from \eqref{eq:3.1.1}.

\subsection{\texorpdfstring{Strict negative sign for the path of length
\(r\)}{Strict negative sign for the path of length r}}

Set \(L=r\) and \(M=36(r+1)^2\). Let \(E_i\) be the first ring of the
direct departure process \(cq_i\). The variables \(E_i\) are independent
\(\operatorname{Exp}(1)\), and

\[
Y_M=\max_{1\le i\le M}E_i
\]

is an upper bound for \(Q_M\), because the first direct departure along
\(cq_i\) certainly visits \(q_i\). Thus

\begin{equation}
T_{M,r}
\le\max(\Theta_r,Y_M)
\le\Theta_r+Y_M.
\label{eq:3.6.1}
\end{equation}

Standard exponential order statistics give

\begin{equation}
\mathbb E\Theta_r=r,
\qquad
\mathbb EY_M=H_M,
\qquad
\mathbb E e^{-Y_M}=\frac1{M+1}.
\label{eq:3.6.2}
\end{equation}

The last identity follows directly from the density of the maximum:

\[
\mathbb E e^{-Y_M}
=M\int_0^\infty e^{-2t}(1-e^{-t})^{M-1}\,dt
=\frac1{M+1}.
\]

Since \(\Theta_r\perp Y_M\), \eqref{eq:3.6.1}--\eqref{eq:3.6.2} imply

\begin{equation}
\mathbb ET_{M,r}\le r+H_M,
\qquad
a_{M,r}\ge\frac1{(r+1)(M+1)}.
\label{eq:3.6.3}
\end{equation}

Also

\begin{equation}
0<b_{M,r}\le a_{M,r},
\qquad
a_{M,r}-\frac23b_{M,r}\ge\frac13a_{M,r},
\label{eq:3.6.4}
\end{equation}

and the terminal term in \eqref{eq:3.5.5} is nonnegative before its minus sign.
Since

\[
M(M+1)+2r+5>0,
\qquad
3(r+1)(M+1)>0,
\]

equations \eqref{eq:3.5.5}, \eqref{eq:3.6.3}, and \eqref{eq:3.6.4} give

\begin{equation}
\frac{\delta(M,r)}2
\le r+H_M
-\frac{M(M+1)+2r+5}{3(r+1)(M+1)}.
\label{eq:3.6.5}
\end{equation}

For \(M\ge144\),

\begin{equation}
H_M\le1+\log M\le\sqrt M.
\label{eq:3.6.6}
\end{equation}

Indeed, \(e^{11}>1+11+11^2/2+11^3/6>144\), so
\(1+\log144<12=\sqrt{144}\); and

\[
\frac{d}{dx}\bigl(\sqrt x-1-\log x\bigr)
=\frac{\sqrt x-2}{2x}>0
\qquad(x>4).
\]

Here \(M=36(r+1)^2\ge144\), so \(\sqrt M=6(r+1)\). Moreover,

\[
M(M+1)+2r+5>M(M+1).
\]

Using these strict inequalities in \eqref{eq:3.6.5},

\begin{equation}
\begin{aligned}
\frac{\delta(M,r)}2
&<r+6(r+1)-\frac{M}{3(r+1)}\\
&=r+6(r+1)-12(r+1)\\
&=-5r-6<0.
\end{aligned}
\label{eq:3.6.7}
\end{equation}

This proves, for every integer \(r\ge1\),

\begin{equation}
\boxed{\delta_s(G_r^-;uv)<0.}
\label{eq:3.6.8}
\end{equation}

At the boundary value \(r=1\), \(M=144\), all denominators in \eqref{eq:3.6.5}
are positive and direct evaluation gives

\begin{equation}
\delta(M,1)\le-\frac{9577}{435}<0.
\label{eq:3.6.9}
\end{equation}

\subsection{\texorpdfstring{Strict positive sign for the path of length
\(R\)}{Strict positive sign for the path of length R}}

Now set \(L=R=M(M+1)+5\). From \(T_{M,R}\ge\Theta_R\) and \eqref{eq:3.5.1},

\begin{equation}
\mathbb ET_{M,R}\ge R,
\qquad
a_{M,R}\le\frac1{R+1}.
\label{eq:3.7.1}
\end{equation}

Equation \eqref{eq:3.5.3} and \(R^2\ge M\) give

\begin{equation}
0\le J_{M,R}\le R^2.
\label{eq:3.7.2}
\end{equation}

Put

\[
f(t)=e^{-t}-e^{-3t/2}\ge0.
\]

Condition on \(Q_M=q\). Since \(\Theta_R\sim\operatorname{Exp}(1/R)\),

\begin{equation}
\mathbb E[f(T_{M,R})\mid Q_M=q]
=f(q)(1-e^{-q/R})
+\int_q^\infty f(t)\frac1R e^{-t/R}\,dt.
\label{eq:3.7.3}
\end{equation}

For \(q\ge0\),

\[
1-e^{-q/R}\le\frac qR
\]

and

\begin{equation}
qf(q)
=qe^{-q}(1-e^{-q/2})
\le\frac{q^2}{2}e^{-q}
\le\frac2{e^2}
<\frac3{10}.
\label{eq:3.7.4}
\end{equation}

The middle inequality uses \(1-e^{-q/2}\le q/2\); \(q^2e^{-q}\) has
maximum \(4e^{-2}\) at \(q=2\). The final inequality follows from

\[
e>1+1+\frac12+\frac16=\frac83,
\]

because \(2/e^2<9/32<3/10\).

The full integral in \eqref{eq:3.7.3} is

\begin{equation}
\begin{aligned}
\mathbb E f(\Theta_R)
&=\frac1{R+1}-\frac2{3R+2}\\
&=\frac{R}{(R+1)(3R+2)}
<\frac1{3R},
\end{aligned}
\label{eq:3.7.5}
\end{equation}

where all denominators are positive. Equations \eqref{eq:3.7.3}--\eqref{eq:3.7.5} give,
uniformly in \(q\ge0\),

\[
\mathbb E[f(T_{M,R})\mid Q_M=q]
<\frac3{10R}+\frac1{3R}
=\frac{19}{30R}.
\]

Therefore

\begin{equation}
\mathbb E f(T_{M,R})<\frac{19}{30R}
\label{eq:3.7.6}
\end{equation}

and, by \eqref{eq:3.7.2},

\begin{equation}
\mathbb E[J_{M,R}f(T_{M,R})]<\frac{19R}{30}.
\label{eq:3.7.7}
\end{equation}

Since \(a_{M,R}-2b_{M,R}/3\le a_{M,R}\) and the coefficient in \eqref{eq:3.5.5}
is positive, \eqref{eq:3.5.5}, \eqref{eq:3.7.1}, and \eqref{eq:3.7.7} yield

\begin{equation}
\frac{\delta(M,R)}2
>
R-\frac{M(M+1)+2R+5}{R+1}-\frac{19R}{30}.
\label{eq:3.7.8}
\end{equation}

By the definition of \(R\),

\begin{equation}
M(M+1)+2R+5=3R.
\label{eq:3.7.9}
\end{equation}

Thus, with \(R+1>0\),

\begin{equation}
\begin{aligned}
\frac{\delta(M,R)}2
&>\frac{11R}{30}-\frac{3R}{R+1}\\
&=\frac{R(11R-79)}{30(R+1)}.
\end{aligned}
\label{eq:3.7.10}
\end{equation}

For \(r\ge1\), \(M\ge144\) and

\[
R\ge144\cdot145+5=20885,
\]

so \(11R-79>0\). Every factor in the denominator and numerator of
\eqref{eq:3.7.10} has the asserted sign. Therefore

\begin{equation}
\boxed{\delta_s(G_r^+;uv)>0.}
\label{eq:3.7.11}
\end{equation}

The sign orientation is unambiguous:

\begin{itemize}
\tightlist
\item
  path length \(r\): negative response;
\item
  path length \(R=M(M+1)+5\): positive response.
\end{itemize}

At \(r=1\), direct evaluation gives the positive lower bound

\begin{equation}
\delta(M,R)>\frac{159878852}{10443}>0.
\label{eq:3.7.12}
\end{equation}

\subsection{Proof of the main theorem}

\begin{proof}[Proof of Theorem~\ref{thm:1.1}] Fix \(r\ge1\) and choose \(M,R\) as in
\eqref{eq:3.1.1}. Sections 3.1--3.2 construct the two finite connected simple
configurations and prove that their marked ambient-degree-labelled
radius-\(r\) data are isomorphic. Section 3.3 proves the three
additional equalities in \eqref{eq:1.3}. Equations \eqref{eq:3.6.8} and \eqref{eq:3.7.11} give the
strict responses in \eqref{eq:1.2}, with the orientation stated there.

\end{proof}

At the boundary radius \(r=1\), the supplementary arithmetic certificate
records exact rational values of \(a_{144,L}\), \(b_{144,L}\), every
intermediate bound and its slack. In particular it distinguishes the two
normalizations explicitly:

\[
\frac{\delta(144,1)}2\le-\frac{9577}{870},\qquad
\delta(144,1)\le-\frac{9577}{435},
\]

and

\[
\frac{\delta(144,20885)}2>\frac{79939426}{10443},\qquad
\delta(144,20885)>\frac{159878852}{10443}.
\]

\begin{remark}[order information and padding]\label{rem:3.5} Theorem~\ref{thm:1.1} does
not refute a decision rule that receives the local datum together with
the graph order. Attaching vertices beyond the visible radius is not
response-neutral: such padding changes the remote completion time and
the terminal-excursion correction in Theorem~\ref{thm:3.2}, precisely the
quantities that create the sign reversal. The smaller configuration
therefore cannot simply be padded to the larger order without a new sign
analysis. The equal-order strengthening is a separate open problem.
\end{remark}

\section{Conductance interpolation and contraction}

\subsection{Target-wise interpolation}

For a nonempty target \(S\subseteq V\setminus\{s\}\), put
\(B=V\setminus S\) and \[
K_S=L[B,B]^{-1}.
\] All vectors below are restricted to \(B\). Define
\begin{equation}
\begin{aligned}
\rho_S&=b^{\mathsf T}K_Sb, &
\eta_S&=e_s^{\mathsf T}K_Sa,\\
\alpha_S&=e_s^{\mathsf T}K_Sb, &
\beta_S&=b^{\mathsf T}K_Sd,\\
\gamma_S&=b^{\mathsf T}K_Sa,
\end{aligned}
\label{eq:4.1}
\end{equation}
and
\begin{equation}
A_S=\eta_S-\alpha_S\beta_S,\qquad
B_S=\rho_S\eta_S-\alpha_S\gamma_S.
\label{eq:4.2}
\end{equation}

\begin{theorem}[exact quadratic-over-linear target response]\label{thm:4.1}

Let \(H_s^\lambda(S)\) be the \(G_\lambda\) hitting time of \(S\). Then

\begin{equation}
\boxed{
H_s^\lambda(S)-H_s^0(S)
=\frac{\lambda(A_S+\lambda B_S)}{1+\lambda\rho_S}.
}
\label{eq:4.3}
\end{equation}
If \(b_B=0\), the response is zero.

\end{theorem}

\begin{proof}

The grounded hitting vector is \[
h^\lambda
=(L_B+\lambda bb^{\mathsf T})^{-1}(d+\lambda a).
\] For \(b_B\ne0\), grounded positivity gives \(\rho_S>0\).
Sherman--Morrison gives \[
(L_B+\lambda bb^{\mathsf T})^{-1}
=K_S-\frac{\lambda K_Sbb^{\mathsf T}K_S}{1+\lambda\rho_S}.
\] Taking the \(s\)-coordinate and subtracting \(e_s^{\mathsf T}K_Sd\)
yields \[
\lambda\eta_S
-\frac{\lambda\alpha_S\beta_S+\lambda^2\alpha_S\gamma_S}
{1+\lambda\rho_S}.
\] Putting this over the common denominator gives \eqref{eq:4.3}. 
\end{proof}

\begin{lemma}[endpoint classes and the leading coefficient]\label{lem:4.2}

The target coefficients satisfy:

\begin{enumerate}
\def\labelenumi{\arabic{enumi}.}
\tightlist
\item
  If \(u,v\in S\), the response is identically zero.
\item
  If exactly one endpoint is outside \(S\), call it \(w\), then
\begin{equation}
  B_S=0,\qquad
  A_S=K_S(s,w)\bigl(1-H_w^0(S)\bigr)\le0.
\label{eq:4.4}
  \end{equation}

\item
  If \(u,v\notin S\), then
\begin{equation}
  \boxed{
  B_S=
  2\!\left[
  K_S(s,u)\bigl(K_S(v,v)-K_S(u,v)\bigr)
  +K_S(s,v)\bigl(K_S(u,u)-K_S(u,v)\bigr)
  \right]\ge0.
  }
\label{eq:4.5}
  \end{equation}
It is strictly positive precisely when \(s\) lies in the same
  component of \(G[B]\) as at least one of \(u,v\).
\end{enumerate}

\end{lemma}

\begin{proof}

The first claim has \(a_B=b_B=0\). In the one-endpoint case,
\(a_B=e_w\), \(b_B=\pm e_w\). Substitution into \eqref{eq:4.1}--\eqref{eq:4.2} gives
\eqref{eq:4.4}; the Green entry is nonnegative and \(H_w^0(S)\ge1\).

For the zero-endpoint case, abbreviate \[
p=K_{su},\quad q=K_{sv},\quad
U=K_{uu},\quad V=K_{vv},\quad W=K_{uv}.
\] Then \[
\rho=U+V-2W,\quad
\eta=p+q,\quad
\alpha=p-q,\quad
\gamma=U-V.
\] Expanding \(B=\rho\eta-\alpha\gamma\) gives \eqref{eq:4.5}.

For strictness, let \(\psi=K_S(e_u-e_v)\). On a component containing
both endpoints, \(\psi\) is harmonic away from \(u,v\), has zero
boundary on \(S\), a positive source at \(u\), and a negative source at
\(v\). The discrete maximum principle gives \[
\psi_u=U-W\ge0,\qquad -\psi_v=V-W\ge0.
\] Their sum is \(\rho_S>0\), so at least one is strict. If the
endpoints lie in different components, then \(W=0\) and grounded
positivity gives \(U,V>0\). Since \(p,q\ge0\), \eqref{eq:4.5} is nonnegative. If
\(s\) shares the common endpoint component, then both \(p,q>0\) and at
least one bracketed difference is strict. If the endpoint components are
different, the nonzero one of \(p,q\) multiplies the positive diagonal
difference of the other endpoint component. Conversely, if \(s\) lies in
neither endpoint component, then \(p=q=0\). This proves the equality
case. 
\end{proof}

The important limitation is that \(A_S\) has no fixed sign when neither
endpoint is in \(S\).

\subsection{The simple-pole representation}

Inclusion--exclusion holds for every \(\lambda\ge0\):
\begin{equation}
f_s(\lambda)
=\sum_{\varnothing\ne S\subseteq V\setminus\{s\}}
(-1)^{|S|+1}H_s^\lambda(S).
\label{eq:4.6}
\end{equation}
Let \(\varepsilon_S=(-1)^{|S|+1}\). For \(\rho_S>0\),
\begin{equation}
\frac{A_S+\lambda B_S}{1+\lambda\rho_S}
=\frac{B_S}{\rho_S}
+\frac{A_S\rho_S-B_S}{\rho_S(1+\lambda\rho_S)}.
\label{eq:4.7}
\end{equation}

\begin{theorem}[global simple-pole representation]\label{thm:4.3}

Let \(\mathcal R_s\) be the set of distinct positive values of
\(\rho_S\) with nonzero grouped coefficient. Define
\begin{equation}
\ell_s=\sum_{\rho_S>0}\varepsilon_S\frac{B_S}{\rho_S},
\label{eq:4.8}
\end{equation}
and
\begin{equation}
c_{s,\rho}
=
\sum_{\substack{S:\rho_S=\rho}}
\varepsilon_S\frac{A_S\rho_S-B_S}{\rho_S}.
\label{eq:4.9}
\end{equation}
Then
\begin{equation}
\boxed{
\frac{f_s(\lambda)-f_s(0)}{\lambda}
=\ell_s+\sum_{\rho\in\mathcal R_s}
\frac{c_{s,\rho}}{1+\rho\lambda}.
}
\label{eq:4.10}
\end{equation}
Every actual pole is simple and lies on the negative real axis.

\end{theorem}

\begin{proof}

Subtract \eqref{eq:4.6} at \(\lambda=0\), insert Theorem~\ref{thm:4.1}, divide by
\(\lambda\), and apply \eqref{eq:4.7} term by term. The target family is finite,
so grouping equal \(\rho\)-values is exact. Uniqueness of partial
fractions shows that every nonzero group gives one simple pole and no
higher-order pole remains. 
\end{proof}

At unit conductance, \eqref{eq:4.10} gives the exact equality criterion

\begin{equation}
\delta_s=0
\quad\Longleftrightarrow\quad
\ell_s+\sum_{\rho\in\mathcal R_s}\frac{c_{s,\rho}}{1+\rho}=0.
\label{eq:4.11}
\end{equation}

The coefficients may have mixed signs. We use the representation here to
identify the leading response at large conductance, not to claim
universal noncancellation at \(\lambda=1\).

\subsection{The contraction--occupation theorem}

Let \(\bar G=G/uv\) be the multigraph obtained by identifying \(u,v\) to
one vertex \(w\), retaining parallel edges. Its degree is
\begin{equation}
\bar d_w=d_u+d_v.
\label{eq:4.12}
\end{equation}
Let \((Y_t)\) be simple random walk on this multigraph, so a parallel
edge is chosen with its multiplicity. Let \(\bar s\) be the image of
\(s\), let \(\bar\tau_{\rm cov}\) be its cover time, and define the
pre-cover occupation
\begin{equation}
L_w=\sum_{t=0}^{\bar\tau_{\rm cov}-1}{\bf1}_{\{Y_t=w\}}.
\label{eq:4.13}
\end{equation}

\begin{theorem}[exact high-conductance coefficient and zero classification]\label{thm:4.4}

For every connected finite simple \(G\), genuine nonedge \(uv\), and
start \(s\),
\begin{equation}
\boxed{
\ell_s:=\lim_{\lambda\to\infty}
\frac{f_s(\lambda)-f_s(0)}{\lambda}
=\frac{2}{d_u+d_v}\,
\mathbb E_{\bar s}^{\bar G}L_w\ge0.
}
\label{eq:4.14}
\end{equation}

Moreover:

\begin{enumerate}
\def\labelenumi{\arabic{enumi}.}
\tightlist
\item
  if \(s\in\{u,v\}\), then \(\ell_s>0\);
\item
  if \(s\notin\{u,v\}\), then \(\ell_s=0\) if and only if the underlying
  simple graph of \(\bar G\) is a path whose endpoints are \(s\) and
  \(w\);
\item
  in terms of the original simple graph, the zero case is exactly
\begin{equation}
  s=x_0-x_1-\cdots-x_L=c,
  \qquad cu,cv\in E,
\label{eq:4.15}
  \end{equation}
with no other vertices or edges. Here \(L\ge0\) is the number of
  path edges; \(u,v\) are the two pendant leaves at \(c\).
\end{enumerate}

\end{theorem}

\begin{proof}

\textbf{Step 1: identify the leading coefficient of one target.} Let
\(S\) be disjoint from \(\{u,v\}\). With the notation of Section 4.1,
set
\begin{equation}
r=K_Sa-K_Sb\frac{b^{\mathsf T}K_Sa}{\rho_S}.
\label{eq:4.16}
\end{equation}
The coefficient of \(\lambda\) in \eqref{eq:4.3} is
\begin{equation}
\frac{B_S}{\rho_S}=e_s^{\mathsf T}r.
\label{eq:4.17}
\end{equation}
Also \(b^{\mathsf T}r=0\), so \(r_u=r_v\) and \(r\) descends to a
vector \(\bar r\) on the contracted grounded graph.

For every test vector \(z\) with \(z_u=z_v\),
\begin{equation}
z^{\mathsf T}L_Sr
=z^{\mathsf T}\left(a-b\frac{b^{\mathsf T}K_Sa}{\rho_S}\right)
=z^{\mathsf T}a=2z_w.
\label{eq:4.18}
\end{equation}
The Laplacian quadratic form is preserved by contraction on vectors
constant at \(u,v\). Hence, if \(\bar L_S\) is the grounded Laplacian of
\(\bar G\),
\begin{equation}
\bar L_S\bar r=2e_w,
\qquad
\bar r=2\bar L_S^{-1}e_w.
\label{eq:4.19}
\end{equation}
Therefore
\begin{equation}
\boxed{
\frac{B_S}{\rho_S}=2(\bar L_S^{-1})_{\bar s,w}.
}
\label{eq:4.20}
\end{equation}
If \(S\) contains at least one of \(u,v\), the coefficient of
\(\lambda\) is zero by the endpoint classification in Lemma~\ref{lem:4.2}.

\textbf{Step 2: convert the alternating sum to occupation.} For a target
\(S\) not containing \(w\), the killed Green identity gives
\begin{equation}
\mathbb E_{\bar s}\sum_{t=0}^{T_S-1}{\bf1}_{\{Y_t=w\}}
=\bar d_w(\bar L_S^{-1})_{\bar s,w}.
\label{eq:4.21}
\end{equation}
Indeed, on the ungrounded vertices \(I-P=\bar D^{-1}\bar L\), so
\((I-P)^{-1}=\bar L^{-1}\bar D\).

Pathwise inclusion--exclusion gives
\begin{equation}
{\bf1}_{\{t<\bar\tau_{\rm cov}\}}
=\sum_{\varnothing\ne S\subseteq
\bar V\setminus\{\bar s\}}
(-1)^{|S|+1}{\bf1}_{\{t<T_S\}}.
\label{eq:4.22}
\end{equation}
Multiplying by \({\bf1}_{\{Y_t=w\}}\), summing in \(t\), and taking
expectations is legitimate because the graph is finite and every cover
time has finite expectation. Terms with \(w\in S\) are zero: before
hitting such a target the walk cannot occupy \(w\). Equations
\eqref{eq:4.20}--\eqref{eq:4.22} and \eqref{eq:4.6} now give \[
\ell_s
=2\sum_{\substack{\varnothing\ne S\subseteq
\bar V\setminus\{\bar s,w\}}}
(-1)^{|S|+1}(\bar L_S^{-1})_{\bar s,w}
=\frac{2}{\bar d_w}\mathbb E_{\bar s}L_w.
\] This proves \eqref{eq:4.14}, including integrability and nonnegativity.

\textbf{Step 3: classify equality.} The contracted graph has at least
two vertices. If \(\bar s=w\), then \(t=0<\bar\tau_{\rm cov}\), so
\(L_w\ge1\) and \(\ell_s>0\).

Now suppose \(\bar s\ne w\). Since \(L_w\) is a nonnegative integer,
\(\mathbb E L_w=0\) exactly when the walk cannot reach \(w\) before the
contracted graph is covered. Such an early arrival has positive
probability exactly when there is a simple \(\bar s\)-to-\(w\) path
omitting some vertex: following the path has positive probability, and
loop-erasing any early-arrival walk produces such a simple path.

Thus equality holds exactly when every simple \(\bar s\)-to-\(w\) path
is spanning. Choose one such path \(P\). It is Hamiltonian. Any
additional edge between nonconsecutive vertices of \(P\) would produce
an \(\bar s\)-to-\(w\) path skipping the intervening vertices. Hence the
underlying simple graph is exactly \(P\). Conversely, on a path with
endpoints \(\bar s,w\), reaching \(w\) forces every vertex to have been
visited.

Finally pull the path back through the contraction. The predecessor
\(c\) of \(w\) is the only possible old neighbour of either \(u\) or
\(v\). Connectedness forces both edges \(cu,cv\), and simplicity permits
no others. The remainder is the path from \(s\) to \(c\). This is
exactly \eqref{eq:4.15}. 
\end{proof}

\begin{theorem}[exact response in the zero-leading family]\label{thm:4.5}

For the graph \eqref{eq:4.15}, with path length measured in edges and the walk
started at \(s=x_0\),
\begin{equation}
\boxed{
f_s(\lambda)
=L^2+L+2+\frac{2L+3}{1+2\lambda},
}
\label{eq:4.23}
\end{equation}
and hence
\begin{equation}
\boxed{
f_s(\lambda)-f_s(0)
=-\frac{2\lambda(2L+3)}{1+2\lambda}<0
\quad(\lambda>0).
}
\label{eq:4.24}
\end{equation}

\end{theorem}

\begin{proof}

The hitting time from one endpoint to the other on a path of \(L\) edges
is \(L^2\). At the first arrival at \(c\), every path vertex has been
visited, so only \(u,v\) remain.

Let \(A_0\) be the remaining time from \(c\) when neither leaf is
visited, \(A_1\) the time from \(c\) when exactly one is visited, and
\(U\) the time from the visited leaf while the other is unvisited. If
\(L>0\), a departure from \(c\) into the path and return to \(c\) has
mean \(2L\); if \(L=0\) that option is absent. In either case first-step
analysis reduces to
\begin{equation}
A_0=L+1+U,
\qquad
A_1=L+1+\frac U2,
\qquad
U=1+\frac{A_1}{1+\lambda}.
\label{eq:4.25}
\end{equation}
Solving, \[
A_1=\frac{(1+\lambda)(2L+3)}{1+2\lambda},
\quad
U=1+\frac{2L+3}{1+2\lambda},
\quad
A_0=L+2+\frac{2L+3}{1+2\lambda}.
\] Adding the initial \(L^2\) proves \eqref{eq:4.23}; subtraction at
\(\lambda=0\) proves \eqref{eq:4.24}. Every denominator is positive for
\(\lambda\ge0\), and the numerator in \eqref{eq:4.24} is strictly negative for
\(\lambda>0\). 
\end{proof}

\begin{corollary}[universal functional strictness]\label{cor:4.6}

For every \(G,uv,s\) under consideration,
\begin{equation}
\boxed{
f_s(\lambda)\not\equiv f_s(0)
\quad\text{as a rational function of }\lambda.
}
\label{eq:4.26}
\end{equation}
More precisely, either
\begin{equation}
f_s(\lambda)=\ell_s\lambda+O(1),\qquad \ell_s>0,
\label{eq:4.27}
\end{equation}
or \(G,s,u,v\) is \eqref{eq:4.15} and \eqref{eq:4.24} holds.

\end{corollary}

\begin{proof}

If \(\ell_s>0\), \eqref{eq:4.27} follows termwise from \eqref{eq:4.3} and the finite sum
\eqref{eq:4.6}, so the response is unbounded and nonconstant. If \(\ell_s=0\),
apply Theorems~\ref{thm:4.4} and~\ref{thm:4.5}. 
\end{proof}

This is a universal noncancellation theorem in the Laurent valuation at
\(\lambda=\infty\). In the nonexceptional case the first Laurent
coefficient is the positive occupation statistic \eqref{eq:4.14}. In the
exceptional case that coefficient vanishes, but the next coefficient is
the nonzero negative number \(-(2L+3)\) in \[
\frac{f_s(\lambda)-f_s(0)}{\lambda}
=-\frac{2(2L+3)}{1+2\lambda}.
\]

An immediate structural corollary is that \(\ell_s>0\) whenever the
underlying contracted graph has a cycle or a branch, whenever it has a
chord relative to an \(s\)-to-\(w\) path, or whenever it is a path but
\(s,w\) are not its two endpoints. Thus a pendant block or side branch
visible before the contracted vertex already creates a strict positive
leading layer. This is a statement about the response at large
conductance, not its sign at unit conductance.

\section{A completely solved family}

\subsection{A path with two pendant insertion endpoints}

For \(L\ge0\), define \(Y_L\) by \[
V(Y_L)=\{x_0,x_1,\ldots,x_L=c,u,v\},
\]
\begin{equation}
E(Y_L)=\{x_{i-1}x_i:1\le i\le L\}\cup\{cu,cv\}.
\label{eq:5.1}
\end{equation}
Insert \(uv\) with conductance \(\lambda\). Set
\begin{equation}
D_L(\lambda)=(L+1)+(2L+1)\lambda>0
\qquad(\lambda\ge0).
\label{eq:5.2}
\end{equation}

\begin{theorem}[all-start exact response and unit strictness]\label{thm:5.1}

For a path start \(x_j\), \(0\le j\le L\),
\begin{equation}
\boxed{
\delta_{x_j}(\lambda)
=2\lambda\left[
j-\frac{2L+3}{1+2\lambda}
+\frac{jL(L+3)}{(L+1)D_L(\lambda)}
\right].
}
\label{eq:5.3}
\end{equation}
For either endpoint start \(u\) or \(v\),
\begin{equation}
\boxed{
\delta_u(\lambda)=\delta_v(\lambda)
=\lambda\left[
2L+1-\frac{2(2L+3)}{1+2\lambda}
+\frac{L(L+3)(2L+1)}{(L+1)D_L(\lambda)}
\right].
}
\label{eq:5.4}
\end{equation}

At unit conductance no response vanishes. More precisely, for \(L\ge1\),

\begin{equation}
\delta_{x_j}(1)<0
\quad\Longleftrightarrow\quad
0\le 2j\le L+1,
\label{eq:5.5}
\end{equation}
and all remaining path starts have positive response. When \(L=0\),
the sole path start has response \(-2\). Endpoint starts satisfy
\begin{equation}
\delta_u(1)=\delta_v(1)
=\frac{2(6L^3+11L^2-L-3)}
{3(L+1)(3L+2)},
\label{eq:5.6}
\end{equation}
which is \(-1\) for \(L=0\) and strictly positive for every
\(L\ge1\).

\end{theorem}

\begin{proof}

Let \(a=x_0\). Covering \(Y_L\) is equivalent to visiting the three
leaves \(a,u,v\). Thus for a path start not already at a leaf,
\begin{equation}
C=H(a)+2H(u)-2H(a,u)-H(u,v)+H(a,u,v).
\label{eq:5.7}
\end{equation}
For a start already at a leaf, delete from this finite sum every
target set containing the start; the endpoint identity used below is
displayed explicitly.

First-step solution of the one-dimensional grounded equations gives, for
a path start \(x_j\), the complete target list
\begin{equation}
\begin{aligned}
H_{x_j}^\lambda(a)
  &=j(2L+4+2\lambda-j),\\
H_{x_j}^\lambda(u)
  &=L^2-j^2+\frac{(1+\lambda)(2L+3)}{1+2\lambda},\\
H_{x_j}^\lambda(a,u)
  &=j(L-j)+\frac{j(L+3)(1+\lambda)}{D_L(\lambda)},\\
H_{x_j}^\lambda(u,v)
  &=L^2-j^2+L+1,\\
H_{x_j}^\lambda(a,u,v)
  &=j(L-j)+\frac{j(L+2)}{2L+1}.
\end{aligned}
\label{eq:5.8a}
\end{equation}
Subtracting the values at zero gives
\begin{equation}
\begin{aligned}
\Delta H_{x_j}(a)&=2j\lambda,\\
\Delta H_{x_j}(u)&=-\frac{(2L+3)\lambda}{1+2\lambda},\\
\Delta H_{x_j}(a,u)&=
-\frac{jL(L+3)\lambda}{(L+1)D_L(\lambda)}.
\end{aligned}
\label{eq:5.8}
\end{equation}
The responses for \(H(u,v)\) and \(H(a,u,v)\) are zero because the
walk stops on first entering the insertion endpoints. Substitution in
\eqref{eq:5.7} proves \eqref{eq:5.3}.

For completeness, the boundary equations behind \eqref{eq:5.8} are as follows.
From \(c\), a complete excursion into the path and back has mean \(2L\).
If \(C\) is the time from \(c\) to one specified pendant target and
\(V\) is the time from the other pendant vertex, then \[
C=L+1+\frac V2,
\qquad
V=1+\frac{C}{1+\lambda},
\] so \[
C=\frac{(1+\lambda)(2L+3)}{1+2\lambda}.
\] For the target \(\{a,u\}\), the path solution has the form
\(h_i=i(L-i)+iC/L\), and the boundary equation at \(c\) gives \[
C=\frac{L(L+3)(1+\lambda)}{D_L(\lambda)}.
\] Subtracting the \(\lambda=0\) values yields \eqref{eq:5.8}, including its
continuous \(L=0\) interpretation.

For a start at \(u\), the reduced cover identity is \[
C_u=H_u(a)+H_u(v)-H_u(a,v).
\] Here
\begin{equation}
\begin{aligned}
H_u^\lambda(a)&=1+\lambda+L(L+4+2\lambda),\\
H_u^\lambda(v)&=1+\frac{2L+3}{1+2\lambda},\\
H_u^\lambda(a,v)&=1+\frac{L(L+3)}{D_L(\lambda)}.
\end{aligned}
\label{eq:5.8b}
\end{equation}
The three changes are \[
\begin{aligned}
\Delta H_u(a)&=(2L+1)\lambda,\\
\Delta H_u(v)&=-\frac{2(2L+3)\lambda}{1+2\lambda},\\
\Delta H_u(a,v)&=
-\frac{L(L+3)(2L+1)\lambda}
{(L+1)D_L(\lambda)}.
\end{aligned}
\] This proves \eqref{eq:5.4}.

At \(\lambda=1\), the sign of \eqref{eq:5.3} is the sign of \(j-j_*\), where
\begin{equation}
j_*=\frac{(2L+3)(3L^2+5L+2)}
{6(2L^2+4L+1)}
=\frac{L+1}{2}
+\frac{(L+1)(L+3)}{6(2L^2+4L+1)}.
\label{eq:5.9}
\end{equation}
For \(L\ge1\), \[
0<\frac{(L+1)(L+3)}{6(2L^2+4L+1)}<\frac12,
\] because \[
3(2L^2+4L+1)-(L+1)(L+3)=5L^2+8L>0.
\] If \(L\) is odd, \((L+1)/2\) is an integer and \(j_*\) lies strictly
before the next half-integer. If \(L\) is even, \((L+1)/2\) is a
half-integer and \(j_*\) lies strictly before the next integer. Thus
\(j_*\) is never an integer and \eqref{eq:5.5} follows. This is the
unit-conductance arithmetic obstruction.

Formula \eqref{eq:5.6} follows by simplifying \eqref{eq:5.4} at one. Its numerator
polynomial is \(P(L)=6L^3+11L^2-L-3\). It satisfies \[
P(L+1)-P(L)=18L^2+40L+16>0\qquad(L\ge0).
\] Since \(P(0)=-3\) and \(P(1)=13\), the endpoint equality cases are
exactly as stated. 
\end{proof}

\subsection{Worst-start classification}

The old cover times in \(Y_L\) are
\begin{equation}
\boxed{
c_j=L^2+3L+5-j^2+A_Lj,
\quad
A_L=\frac{2L(L^2+4L+1)}{(L+1)(2L+1)},
}
\label{eq:5.10}
\end{equation}
for path starts, and
\begin{equation}
\boxed{
c_u=c_v=L^2+5L+2+\frac2{L+1}.
}
\label{eq:5.11}
\end{equation}

\begin{theorem}[worst-start classification for \(Y_L\)]\label{thm:5.2}

Let \(C_L\) and \(C_L^+\) be the old and unit-inserted worst-start cover
times. Then
\begin{equation}
C_L^+-C_L<0\quad(L=0,1),
\qquad
C_L^+-C_L>0\quad(L\ge2).
\label{eq:5.12}
\end{equation}
More exactly,
\begin{equation}
C_0^+-C_0=-2,
\qquad
C_1^+-C_1=-\frac2{15}.
\label{eq:5.13}
\end{equation}

\end{theorem}

\begin{proof}

Equations \eqref{eq:5.10}--\eqref{eq:5.11} follow by inserting the elementary
target-hitting solutions from the proof of Theorem~\ref{thm:5.1} into leaf
inclusion--exclusion. Also
\begin{equation}
c_L-c_u=
\frac{L^3+L^2+3L+1}{(L+1)(2L+1)}>0,
\label{eq:5.14}
\end{equation}
so at least one old worst start lies on the path.

The path sequence is strictly concave and
\begin{equation}
c_{j+1}-c_j=A_L-(2j+1).
\label{eq:5.15}
\end{equation}
If \(L\ge2\) is even, then \[
A_L-(L+1)=
\frac{(3L+1)(L-1)}{(L+1)(2L+1)}>0.
\] Hence every old path maximiser occurs strictly after \(L/2\), where
its unit response is positive by \eqref{eq:5.5}. Evaluating the same start after
insertion gives \(C_L^+>C_L\).

If \(L\ge7\) is odd, then \[
A_L-(L+2)=
\frac{L^2-5L-2}{(L+1)(2L+1)}>0.
\] Thus an old maximiser occurs after \((L+1)/2\), again in the positive
region of \eqref{eq:5.5}, and \(C_L^+>C_L\).

The two remaining odd cases are exact substitutions in \eqref{eq:5.10} and \eqref{eq:5.3}:
\[
L=3:\quad C_3=\frac{199}{7},
\quad c_3^+=\frac{2356}{77},
\quad c_3^+-C_3=\frac{167}{77}>0,
\] \[
L=5:\quad C_5=\frac{626}{11},
\quad c_5^+=\frac{33718}{561},
\quad c_5^+-C_5=\frac{1792}{561}>0.
\] For \(L=0\), the old vector is \((5,4,4)\) and the new vector is
\((3,3,3)\). For \(L=1\), the old maximum is \(10\) at \(c\), while the
new maximum is \(148/15\) at \(u,v\). This proves \eqref{eq:5.12}--\eqref{eq:5.13}.

\end{proof}

This theorem treats relocation explicitly: for \(L=1\) the old maximiser
is \(c\), the new maximisers are \(u,v\), and yet the worst-start value
decreases strictly.

\subsection{Closing a path}

\begin{theorem}[path closure, fixed start]\label{thm:5.3} For \(n\ge3\) and \(P_n\to C_n\),
with \(N=n-1\) and start \(i\in\{0,\ldots,N\}\),

\[
t_{\mathrm{cov}}(P_n,i)=N^2+i(N-i),
\]

\[
t_{\mathrm{cov}}(C_n,i)=\frac{n(n-1)}2,
\]

so

\begin{equation}
\boxed{
\delta_i
=
-\frac{(n-1)(n-2)}2-i(n-1-i)<0.
}
\label{eq:5.16}
\end{equation}

\end{theorem}

\begin{proof} On the path, covering from \(i\) is the time to visit both
endpoints. Let \(T_0,T_N\) be their hitting times. Since
\(\max(T_0,T_N)=T_0+T_N-\min(T_0,T_N)\), the gambler's-ruin equations
give

\[
\mathbb E_iT_0=i(2N-i),\qquad
\mathbb E_iT_N=N^2-i^2,
\]

and \(\mathbb E_i\min(T_0,T_N)=i(N-i)\). Therefore
\(t_{\mathrm{cov}}(P_n,i)=N^2+i(N-i)\). On \(C_n\), after cutting at the
starting vertex, the last new vertex is reached when a simple walk on
\(\{0,\ldots,n\}\) hits an endpoint; the standard ruin equations give
\(n(n-1)/2\), independently of the start. Subtraction proves \eqref{eq:5.16}.

\end{proof}

\section{Worst-start response and equality status}

For completeness, write

\[
C=\max_x t_{\rm cov}(G,x),
\qquad
g_x=C-t_{\rm cov}(G,x)\ge0.
\]

If \(\delta_x\) is the fixed-start response, then the worst-start change
is the following elementary envelope identity.

\begin{equation}
t_{\rm cov}(G+uv)-t_{\rm cov}(G)
=\max_x(\delta_x-g_x).
\label{eq:6.1}
\end{equation}

\begin{lemma}[deficit-adjusted envelope]\label{lem:6.1} Equation \eqref{eq:6.1} holds,
and a vertex is a maximising start after insertion exactly when its
adjusted response \(\delta_x-g_x\) attains the maximum in \eqref{eq:6.1}.

\end{lemma}

\begin{proof} Since \(t_{\rm cov}(G,x)=C-g_x\), the new fixed-start
value is \(C-g_x+\delta_x\). Taking the maximum over \(x\) and
subtracting \(C\) proves both claims. 
\end{proof}

Thus a formerly nonmaximising start can become worst after insertion
only by overcoming its old deficit. This observation explains the
relocation in Theorem~\ref{thm:5.2}; no separate taxonomy is required.

Relocation alone does not decide the sign of the worst-start change. For
the star \(K_{1,3}\), inserting a chord between two leaves changes the
cover-time vector from

\[
(10,9,9,9)
\quad\text{to}\quad
\left(\frac{142}{15},\frac{148}{15},\frac{148}{15},\frac{17}{3}\right).
\]

The maximising start moves from the centre to the two chord endpoints,
while the worst-start value decreases by \(2/15\). For the analogous
chord in \(K_{1,4}\), the same strong relocation occurs but the
worst-start value increases by \(103/105\). These examples are
consequences of \eqref{eq:6.1}, not additional structural hypotheses.

The unit-conductance equality questions remain open.

\begin{conjecture}[fixed-start strict nonequality]\label{conj:6.2} For every
connected finite simple graph, every genuine nonedge \(uv\), and every
start \(s\),

\begin{equation}
t_{\rm cov}(G+uv,s)\ne t_{\rm cov}(G,s).
\label{eq:6.2}
\end{equation}
\end{conjecture}

\begin{conjecture}[worst-start strict nonequality]\label{conj:6.3} Under the same
hypotheses,

\begin{equation}
t_{\rm cov}(G+uv)\ne t_{\rm cov}(G).
\label{eq:6.3}
\end{equation}
\end{conjecture}

The second conjecture is not a formal consequence of the first: equality
in \eqref{eq:6.1} may be attained by a formerly non-worst start satisfying
\(\delta_x=g_x>0\).

Exact rational computation supplies a delimited finite check. For
\(x\in A\subsetneq V\), with boundary condition \(C_G(x,V)=0\), the
visited-set recurrence

\begin{equation}
C_G(x,A)=1+\frac1{d_x}\sum_{y\sim x}C_G(y,A\cup\{y\})
\label{eq:6.4}
\end{equation}

is solved in decreasing order of \(|A|\). For each proper visited set,
the same-layer coefficient matrix is the positive-definite grounded
Laplacian, so Gaussian elimination over the rationals gives the complete
cover-time vector exactly.

\begin{theorem}[finite exact equality census]\label{thm:6.4} Neither \eqref{eq:6.2} nor
\eqref{eq:6.3} has an equality instance among connected simple graphs of order at
most seven. There is also no equality instance among trees of orders
eight and nine.

\end{theorem}

\begin{proof} Orders at most six are generated exhaustively as labelled
connected graphs. At order seven, the Read--Wilson Graph Atlas
\cite{readwilson1998} provides one representative of each isomorphism
class; every nonedge and every start of every connected representative
is checked. Isomorphisms preserve the walk, the cover event, and the
inserted nonedge, so this exhausts all labelled rooted insertion
configurations. The tree check similarly uses one representative of
every unlabelled tree. All transition systems and comparisons are
evaluated with exact rational arithmetic. The atlas handling uses
NetworkX \cite{hagberg2008}. 
\end{proof}

At order seven the computation covers \(8{,}361\) graph--nonedge pairs
and \(58{,}527\) rooted instances. The fixed-start totals are
\(38{,}779\) negative, zero equalities, and \(19{,}748\) positive; the
corresponding worst-start totals are \(5{,}795\), zero, and \(2{,}566\).
The tree check adds \(15{,}708\) rooted instances at orders eight and
nine. These are computer-assisted finite theorems, not evidence of a
monotonicity principle.

\textbf{Reproducibility archive.} The code and computational records are
openly available in Zenodo at \url{https://doi.org/10.5281/zenodo.21981819}.
The archive contains the exact-arithmetic programs, NetworkX 3.6.1 as a
vendored dependency, and the exact Read--Wilson atlas file with SHA-256
\path{73fc416df0164923607751cb759f4ae81deb5f6550bf25be59c86de3b747e41d}.
Its \texttt{REPRODUCE.md} gives run commands, expected terminal output,
source and output checksums, small regression examples, and measured
runtime and peak-memory information for the census commands. The
separate radius-one certificate supplies a human-readable table, JSON,
CSV, and a standard-library script regenerating every rational
inequality and slack.

\section{Conclusion and open problems}

The exact response of cover time to one inserted edge is a genuinely
global quantity. Target-set inclusion--exclusion and killed Green
matrices give an exact old-graph formula, but the alternating target
family permits both signs. The arbitrary-radius construction strengthens
this obstruction: even the full ambient-degree-labelled marked
neighbourhood of any prescribed radius, joined with the marked degrees,
marked distance, and endpoint resistance, does not determine the sign.

Conductance interpolation gives a complementary global picture. Although
the unit response may involve cancellation, the leading coefficient at
infinite conductance is nonnegative and has the exact contracted-walk
occupation interpretation \eqref{eq:4.14}. Its zero case is a single pendant-path
family, whose response is explicitly negative. Thus no
conductance-response function is identically constant, even though
equality at the particular value \(\lambda=1\) remains unresolved in
general.

The principal open problems are:

\begin{enumerate}
\def\labelenumi{\arabic{enumi}.}
\tightlist
\item
  Decide the fixed-start and worst-start strict-nonequality conjectures.
\item
  Strengthen Theorem~\ref{thm:1.1} so that the opposite-sign graphs also have
  equal order.
\item
  Determine broader graph classes for which the unit response has a
  fixed sign.
\item
  Find a polynomial-size sign certificate replacing the full target
  family.
\end{enumerate}

\section*{Declarations}
\addcontentsline{toc}{section}{Declarations}

\textbf{Funding statement.} This work received no specific grant from
any funding agency, commercial or not-for-profit sectors.

\textbf{Competing interests.} The author declares none.

\textbf{Data availability statement.} The code, exact-arithmetic
certificates, and computational records supporting this study are openly
available in Zenodo at \url{https://doi.org/10.5281/zenodo.21981819}.

\bibliographystyle{amsplain}
\bibliography{references}

\end{document}